\documentclass[a4paper,10pt]{article}
\usepackage{amsmath,amssymb}
\usepackage[utf8]{inputenc}
\usepackage{latexsym, graphicx, xcolor, amsthm}
\usepackage[colorlinks,citecolor=blue]{hyperref}

\def\QQ{{\mathbb Q}}
\def\Q{{\mathcal Q}}

\def\J{{\mathcal J}}
\def\R{{\mathcal R}}
\def\P{{\mathcal P}}

\def\A{{\mathcal A}}
\def\B{{\mathcal B}}

\definecolor{green2}{rgb}{0.18, 0.545, 0.341}  

\newcommand{\N}{\mathbb{N}}
\newcommand{\M}{\mathbb{M}}
\newcommand{\sI}{\mathcal{I}}

\newcommand{\li}{\mathcal{L}}
\newcommand{\rk}{\operatorname{rk}}

\newtheorem{theo}{Theorem}[]
\newtheorem{prop}[theo]{Proposition}
\newtheorem{lemm}[theo]{Lemma}
\newtheorem{coro}[theo]{Corollary}

\title{Poset structures on $(m+2)$-angulations and polynomial bases of
  the quotient by $G^m$-quasisymmetric functions}

\author{Jean-Christophe Aval \& Frédéric Chapoton}

\date{\today}

\begin{document}
\maketitle


\begin{abstract}
  For integers $m, n \geq 1$, we describe a bijection sending
  dissections of the $(m n + 2)$-regular polygon into $(m+2)$-sided
  polygons to a new basis of the quotient of the polynomial algebra in
  $m n$ variables by an ideal generated by some kind of higher
  quasi-symmetric functions. We show that divisibility of the basis
  elements corresponds to a new partial order on dissections, which is
  studied in some detail.
\end{abstract}

\section{Introduction}

Let $m\geq 1$ be an integer. For every integer $n \geq 1$, we define a
simple poset structure $\P_{m,n}$ on the set of $(m+2)$-angulations of
a $(m n+2)$-gon. This generalizes the construction by Pallo on
triangulations \cite{pallo03}, which is closely related to the Tamari
lattice. 

Quasisymmetric functions, a generalisation of symmetric functions,
were introduced in \cite{gessel84} and are now classical in algebraic
combinatorics. Some higher analogues were introduced in
\cite{poirier1998} and further studied in \cite{BH08}. We recall
some results about quotients of polynomials rings by higher
quasi-symmetric functions, first obtained for quasisymmetric functions in
\cite{ab2003,abb2004} and extended to the higher case in \cite{aval06}.

We then show that the poset $\P_{m,n}$ is isomorphic to the
divisibility poset of a new particular basis of the quotient of the
polynomial ring in $m$ sets of $n$ variables by the ideal generated by
$G^m$-quasisymmetric functions without constant term. Our description
of a new basis builds upon the basis indexed by $m$-Dyck paths that was
introduced for general $m$ in \cite{aval06}.

For $m=1$, the posets $\P_{1,n}$ were introduced by Pallo
in \cite{pallo03} and further studied in \cite{csar12,csar14}, and the
basis indexed by triangulations was defined in \cite{chapoton05}, but
the connection between them is new.

The last two sections of the article are devoted to some enumerative
results on the posets (enumeration of intervals, rank generating
function) and to a recursive description of the intervals as distributive
lattices of orders ideals of forests.

\section{A poset structure on $M$-angulations}

Let $m\geq 1$ be an integer. For the sake of readability, we will use
the expression $M$-angulation instead of $(m+2)$-angulation. An
$M$-angulation of a regular convex polygon is a set of diagonals that
cut the polygon into $(m+2)$-sided polygonal regions.

We consider the set $\Q_{m,n}$ of $M$-angulations of a $(m n+2)$-gon.
Every element of $\Q_{m,n}$ contains $n$ regions, separated
by $n-1$ diagonal edges.  The cardinality of $\Q_{m,n}$ is given by
$\frac{1}{m n+1}\binom{(m+1)n}{n}$, the number of $(m+1)$-ary planar
rooted trees with $n$ inner vertices (often called a Fuss-Catalan number). A simple bijection between these
two classes of objects is given by planar duality. Some elements of
$\Q_{2,7}$ are shown in Figures \ref{fig:quad} and \ref{fig:fan}.
\begin{figure}[ht!] 
\centering \includegraphics[scale=0.5]{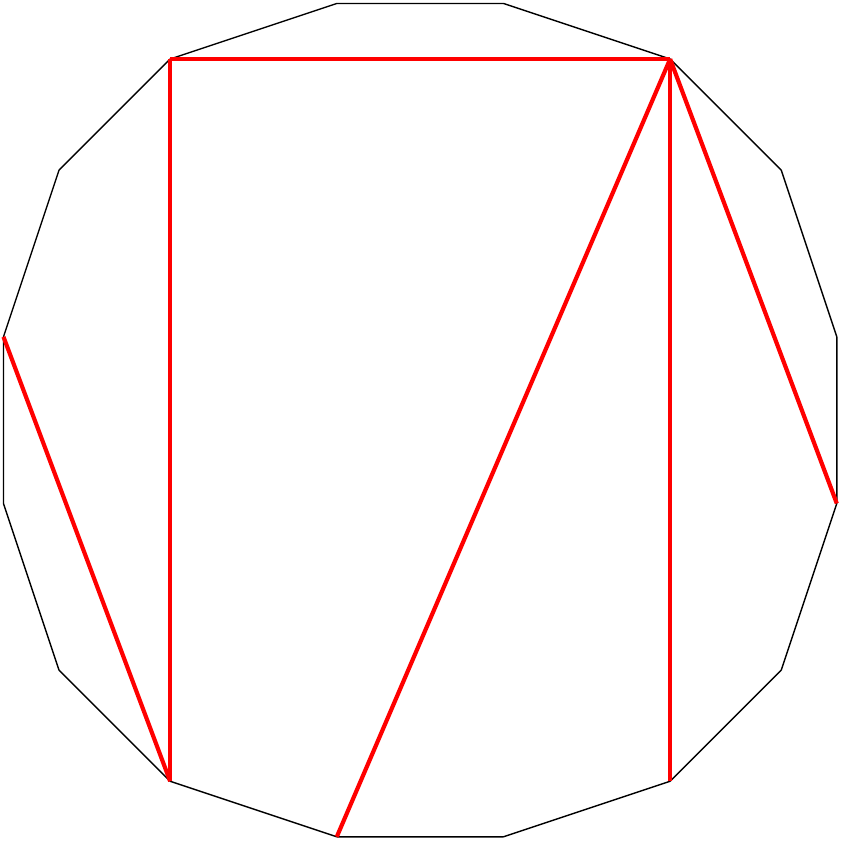} 
\caption{A quadrangulation of the 16-gon}
\label{fig:quad}
\end{figure}

To define a poset structure on $\Q_{m,n}$, we fix a particular element $Q_0$,
which is a fan (every diagonal edge involves a fixed vertex, denoted $0$ and called the \textit{apex},
see Figure \ref{fig:fan}).
\begin{figure}[ht!] 
\centering \includegraphics[scale=0.5]{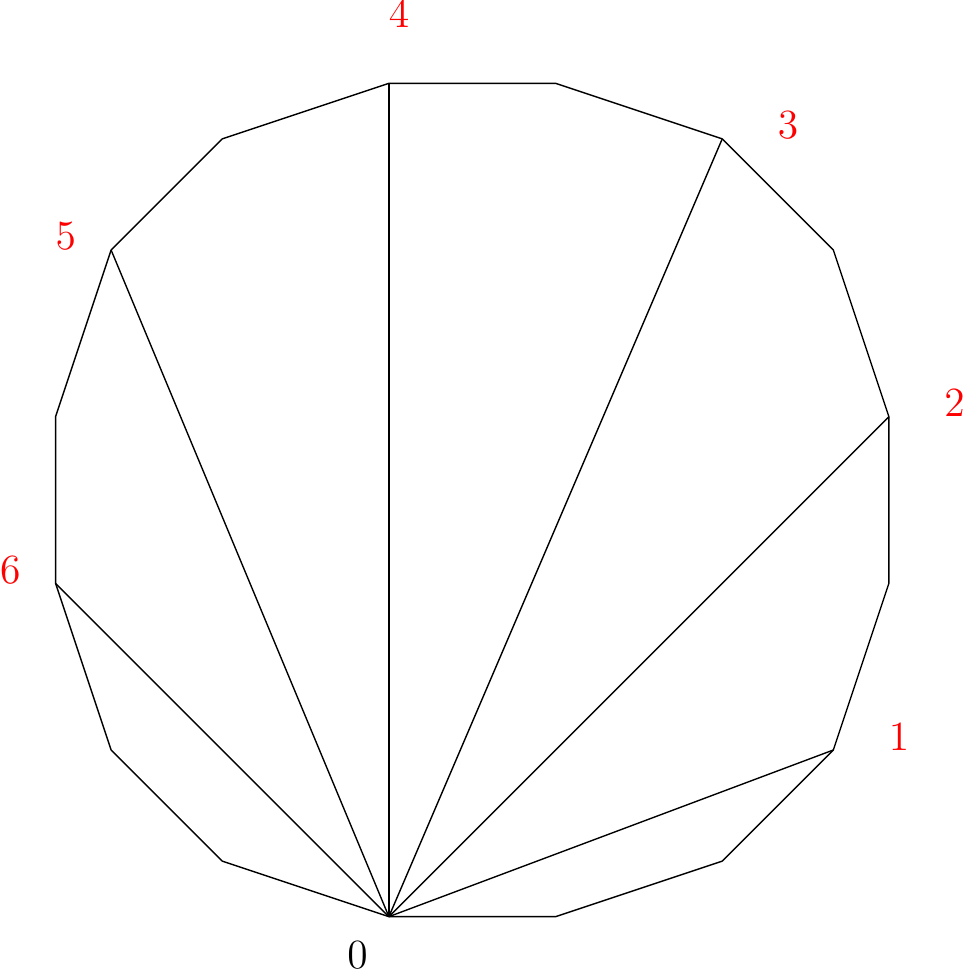}
\caption{The fan $Q_0$ for $m=2$ and $n=7$}
\label{fig:fan}
\end{figure}

We consider the following order relation.  An element $Q$ of
$\Q_{m,n}$ is covered in $\P_{m,n}$ by the $M$-angulations obtained by
flipping one of its diagonal edges included in $Q_0$. Here flipping
means removing this diagonal edge and replacing it by another diagonal
edge cutting again the $(2 m+2)$-gon created by the removal into two
$(m+2)$-gons. Note that any diagonal edge may be flipped in exactly
$m$ different ways.  As an example the poset $\P_{2,3}$ is shown on
Figure \ref{fig:P3}.
\begin{figure}[ht!] 
  \centering \includegraphics[scale=1.2]{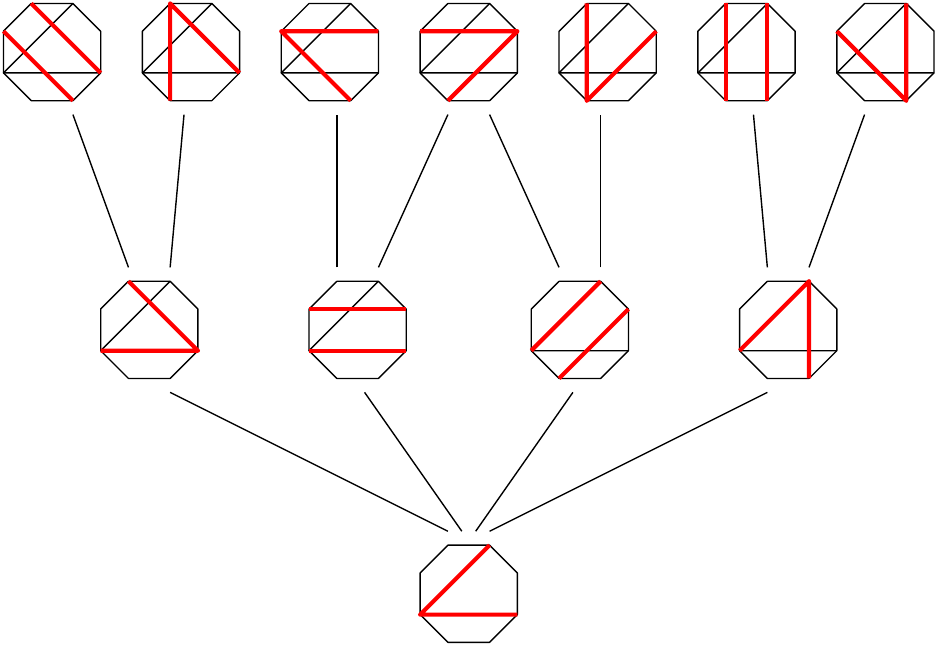} 
  \caption{The poset $\P_{2,3}$}
  \label{fig:P3}
\end{figure}

This poset is clearly graded, the rank of an element $Q$ being $n-1$
minus the number of diagonal edges shared by $Q_0$ and $Q$. This is also
the number of diagonal edges in $Q$ that are not in $Q_0$.

It follows from the description of the coverings that an element of
rank $r$ is covered by exactly $m(n-1-r)$ elements. This implies the
formula $m^{n-1}(n-1)!$ for the number of maximal chains in
$\P_{m,n}$.

This poset has $Q_0$ as unique minimum. The fact that $Q_0$ is smaller
than all $M$-angulations follows from the next lemma by induction on
the number of diagonals that are both in $Q$ and $Q_0$. This lemma
implies that unless $Q$ is $Q_0$, there is a cover relation
$Q' \triangleleft Q$ where $Q'$ has one more diagonal in common with
$Q_0$.

\begin{lemm}
  \label{descente_possible}
  Let $Q$ be distinct from $Q_0$. Then there is always at least one
  diagonal $d_0$ in $Q_0$ that cuts exactly one diagonal of $Q$.
\end{lemm}
\begin{proof}


  One can assume without restriction that $Q$ and $Q_0$ have no common
  diagonal, otherwise one can find $d_0$ (by induction) inside one of
  the parts cut by the common diagonals. In at least one of these
  parts the restriction of $Q$ must differ from the restriction of
  $Q_0$, because $Q$ is not $Q_0$.

  Let us label all the vertices of the regular polygon
  counter-clockwise by integers, starting from the vertex $0$ which is
  the apex of $Q_0$.

  Because $Q$ and $Q_0$ have no common diagonal, there exists a unique
  region $R_0$ of $Q$ that contains the vertex $0$ in its
  boundary. Removing $R_0$ from the ambient polygon, one gets one or
  more convex polygons, all of them with $k m+2$ vertices for some
  $k$. Let us choose one of these polygons, and let $Q'$ be its
  $M$-angulation obtained from $Q$ by restriction. Let $R'$ be the
  unique region of $Q'$ which is adjacent to $R_0$.

  Excluding minimal and maximal indices, vertices of $Q'$
  form a sequence of $m k$ vertices numbered consecutively, in which
  every residue class modulo $m$ is represented exactly $k$ times.

  Because of the tree-like structure of $M$-angulations, one can build
  the $M$-angulation $Q'$ by successive additions of $M$-angles,
  starting from the region $R'$. Going backwards, one can go from $Q'$
  to $R'$ by removal of $M$-angles on the boundary.

  Using this leaf-removal induction, one can then prove that the
  boundary of $R'$ (excluding minimal and maximal indices) contains
  exactly one representative of every residue class modulo $m$.

  Note now that in $Q_0$, the vertices linked by a diagonal to the
  apex $0$ form a residue class modulo $m$, when numbered in the same
  way as the vertices of $Q$.

  It follows that exactly one of the vertices of $R'$ (excluding minimal and maximal indices) is the end of a
  diagonal $d_0$ in $Q_0$. The diagonal $d_0$ does only cross one diagonal
  of $Q$, namely the diagonal separating the regions $R_0$ and $R'$.
\end{proof}

The maximal elements in the poset $\P_{m,n}$ are the $M$-angulations that have no common
diagonal with $Q_0$. We will call them \textit{final $M$-angulations}.

Remark: There are several interesting existing families of posets on objects in
bijection with $(m+2)$-angulations, including $m$-Tamari lattices
\cite{bpr2012} and $m$-Cambrian lattices of type $A$
\cite{stw2015}. The posets introduced here seem to be new.

\section{$m$-analogues of $B$-quasisymmetric functions}

In \cite{BH08}, Baumann and Hohlweg introduced the ring of
$B$-quasisymmetric functions as the graded dual Hopf algebra of the
analog in type $B$ of Solomon's descent algebra. This ring is
contained in the polynomial ring in two sets of $n$ variables.

In \cite{aval06}, a quotient ring of this polynomial ring by an ideal
of $B$-quasi\-symmetric functions was studied. Moreover, for every
integer $m \geq 1$, an analog of the ring of $B$-quasisymmetric functions and
an analog of the quotient ring were also defined and studied, involving $m$
sets of $n$ variables.

We refer to \cite{aval06} for the original motivations of the study of
these rings and for the proof of the results that we will use.
Let us now summarize the results of \cite{aval06} in their most general form.

Let us denote by $\M=\{x,y,z,\dots,\omega\}$ a set of $m$ distinct
letters, endowed with a total order $x < y < z < \dots < \omega$.  We
will mostly illustrate our constructions with the cases $m=2$ and
$m=3$, therefore using only the letters $x,y,z$.

Let us start with polynomials in the union of $m$ alphabets of each
$n$ variables: $X_n=x_1,\dots,x_n$, $Y_n=y_1,\dots,y_n$, up to
$\Omega_n=\omega_1,\dots,\omega_n$. Denote this polynomial ring by
$\QQ[X_n,Y_n,\dots,\Omega_n]$. Inside this polynomial ring, one can
define a space of $G^m$-quasisymmetric functions, which reduces when $m=1$ to
the classical quasisymmetric functions.

In \cite{aval06}, a Gröbner basis for the ideal $\J_{m,n}$ generated
by constant-term-free $G^m$-quasisymmetric functions was described, and
from that was deduced a monomial basis for the quotient $\R_{m,n}$ of the
polynomial ring $\QQ[X_n,Y_n,\dots,\Omega_n]$ by $\J_{m,n}$.

This monomial basis for the quotient $\R_{m,n}$ is indexed by $m$-Dyck
paths, which gives the dimension
formula $\dim \R_{m,n}=\frac{1}{m n+1}\binom{(m+1) n}{n}$.

\subsection{Definitions}

For these definitions, we follow \cite{BH08}, with some minor differences, for the sake of simplicity of the computations we will have to make. The main change is to describe directly the general case for any $m \geq 1$ and not the special case $m=2$.

An {\em $m$-vector} of size $n$ is a vector
$v=(v_1,v_2,\dots,v_{m n-1},v_{m n})$ of length $m n$ with entries in
$\N$. One must think of $m$-vectors as the concatenation of $n$
sequences of length $m$. An {\em $m$-composition} is an $m$-vector in
which there is no sequence of $m$ consecutive zeros.

The integer $n$ is called the {\em size} of $v$. The {\em weight} of
$v$ is by definition the $m$-tuple $(w_1,\dots,w_m)$ where
$w_j=\sum_{i=0}^{n-1} v_{m i + j}$. We also set
$|v|=\sum_{i=1}^{m n} v_i$. For example $(1,0,2,1,0,2,3,0)$ is a
$2$-composition of size $4$, and of weight $(6,3)$.

To make notations lighter, we shall sometimes write $m$-vectors or
$m$-compo\-sitions as words with bars, where the bars separates the
word into $n$ blocks of length $m$. For example, $10|21|02|30$ stands
for the $2$-vector $(1,0,2,1,0,2,3,0)$ (see also the following
definition).

\vskip 0.3cm 

Let us now define the {\em fundamental $G^m$-quasisymmetric
  polynomials}, indexed by $m$-compositions. 

Let $c=(c_1,\dots,c_{m n})$ be an $m$-composition. One can decompose
$c$ as a concatenation of $n$ blocks of $m$ integers. Let us first
associate to $c$ a word $w_c$ in the alphabet
$\{x,y,z,\dots,\omega\}$, defined as the concatenation, over all
blocks $b=(b_1,\dots,b_m)$ of $c$, of the word
$x^{b_1} y^{b_2} \dots \omega^{b_m}$. For example, for the
$3$-composition $010|201$, one obtains $w_c = y x x z$ (powers are
written as repeated letters).

Then the fundamental $G^m$-quasisymmetric polynomial of index $c$ is
\begin{equation*}
F_{c}=\sum_{i} \prod_{t \in w_c} t_{i(t)},
\end{equation*}
where the sum is taken over all maps $i$ from the sequence of letters of
the word $w_c$ to the set $\{1,\dots,n\}$ such that $i$ is weakly increasing
inside every block of $c$ and strictly increasing between two blocks.

Let us give some examples for $m=2$:
\begin{align}
  F_{12} &=\sum_{i\le j\le k}x_i y_j y_k,\\
  F_{02|10} &=\sum_{i\le j< k}y_i y_j x_k.
\end{align}

It is clear from the definition that the multidegree ({\it i.e.} the $m$-tuple (degree in $x$, degree in $y$, ..., degree in $\omega$)) of $F_c$ in $\QQ[X_n,Y_n,\dots,\Omega_n]$ is the weight of $c$. If the size of $c$ is greater than $n$, we set $F_c=0$.

The space of $G^m$-quasisymmetric polynomials, denoted by
$QSym_n(G^m)$ is the vector subspace of the ring $\QQ[X_n,Y_n,\dots,\Omega_n]$
generated by the $F_c$, for all $m$-compositions $c$.

Let us denote by $\J_{m,n}$ the ideal $\langle QSym_n(G^m)^+\rangle$
generated by $G^m$-quasi\-symmetric polynomials with zero constant term.

\subsection{A monomial basis for $\R_{m,n}$}

\label{dyck_paths}

Let $v=(v_1,v_2,\dots,v_{m n-1},v_{m n})$ be an $m$-vector of size
$n$. We associate to $v$ a path $\pi(v)$ in the plane $\N\times\N$,
with steps $(0,1)$ (up step) or $(m,0)$ (right step). We start from $(0,0)$ and for each
entry $v_i$ (read from left to right) add $v_i$ right steps $(m,0)$ followed
by one up step $(0,1)$. This clearly defines a bijection between
$m$-vectors of size $n$ and such paths of height $m n$.

As an example, the path associated to the $2$-vector
$(1,0,1,2,0,0,1,1)$ is

\vskip 0.2cm
\centerline{\includegraphics{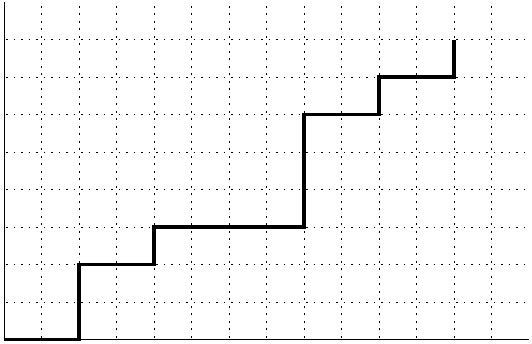}}
\vskip 0.2cm

If the path $\pi(v)$ associated with an $m$-vector $v$ always remains
above the diagonal $x=y$, we call this path an {\em $m$-Dyck path}, and say
that the corresponding $m$-vector $v$ is an {\em $m$-Dyck vector}.

Being an $m$-Dyck vector is equivalent to the condition that, for any
$1\le \ell\le m n$, one has
\begin{equation*}
  m (v_1+v_2+\cdots+v_{\ell}) < \ell.
\end{equation*}

For $v$ an $m$-vector (of length $m n$), we denote by $\A_v$ the monomial 
\begin{equation*}
  \A_v=(x_1^{v_1}y_1^{v_2}\cdots \omega_1^{v_m})(x_2^{v_{m+1}}y_2^{v_{m+2}}\cdots\omega_2^{v_{2m}})\cdots (x_n^{v_{m(n-1)+1}}\cdots\omega_n^{v_{m n}}).
\end{equation*}
This clearly defines a bijection between $m$-vectors and all monomials in the polynomial ring $\QQ[X_n,Y_n,\dots,\Omega_n]$.

For example, the monomial associated to the $2$-vector $(1,0,1,2,0,0,1,1)$ is $x_1 x_2 y_2^2 x_4 y_4$.

The following result was proved in \cite[Th. 5.1]{aval06}.

\begin{prop}
  The set $\B_{m,n}$ of monomials $\A_v$ for $v$ 
  varying over {\em $m$-Dyck vectors} of size $n$ is a basis for the space $\R_{m,n}=\QQ[X_n,Y_n,\dots,\Omega_n]/\J_{m,n}$.
\end{prop}

\section{A bijection between $M$-angulations and $m$-Dyck paths}

We assign to each vertex of the $m n+2$-gon (except the vertex $0$) a
letter in $\M$ as follows. The vertices are labelled by repeating the
sequence $x,y,z,\dots,\omega$ in counter-clockwise order around the
polygon, in such a way that the final vertex just before the
vertex $0$ receives the last letter $\omega$ of $\M$.  It follows that
the first vertex just after the vertex $0$ also receives the letter
$\omega$.

See Figure \ref{fig:diag} for an illustration of this labelling when
$m=2$ with the ordered set of letters $x < y$.
\begin{figure}[ht!] 
  \centering \includegraphics[scale=0.5]{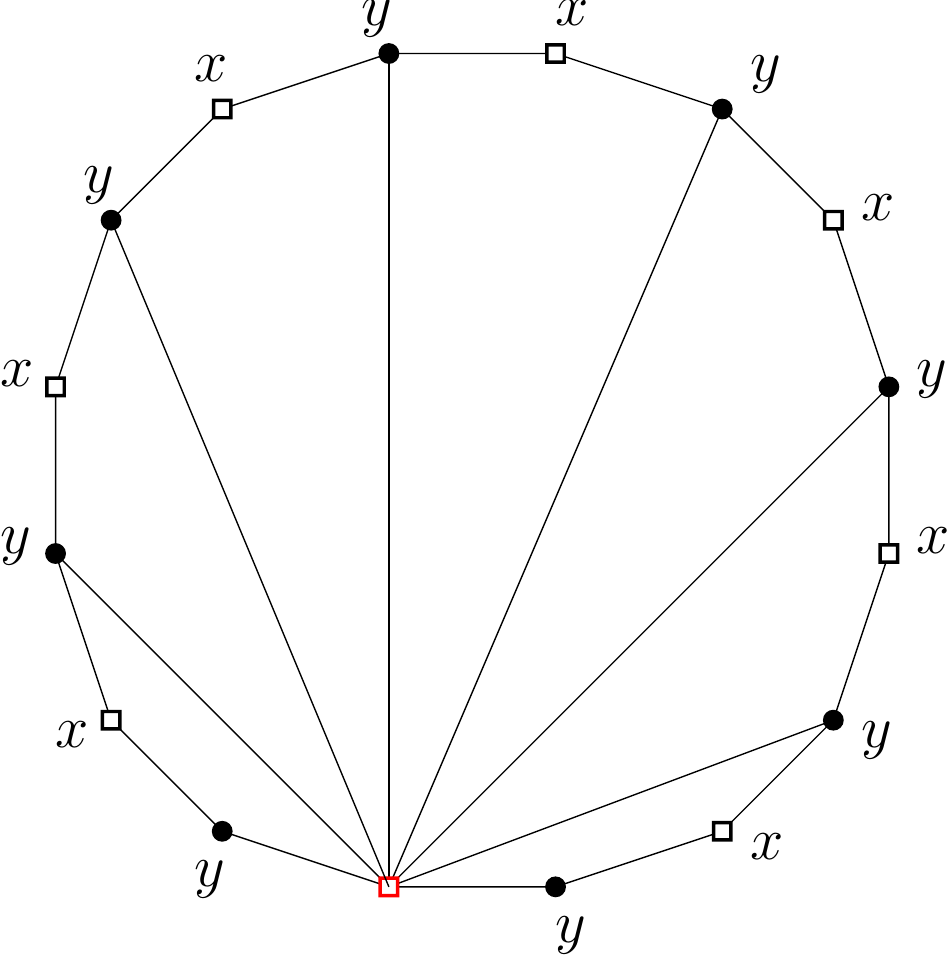} 
  \caption{The labelling of vertices by $x$ and $y$ letters ($m=2$)}
  \label{fig:diag}
\end{figure}

Let us also label the inner diagonals of $Q_0$ from $1$ to $n-1$ in
counter-clockwise order.

\medskip

Then we consider an $M$-angulation $Q$.  To any diagonal $d$ of $Q$,
we associate a polynomial $m_d$.  If $d$ coincides with a diagonal of
$Q_0$, we set $m_d=1$.  Otherwise, and reading counter-clockwise, $d$
starts from a vertex labelled by the letter $u$, intersects
consecutive diagonals of $Q_0$ labelled from $i$ to $j$ and ends at a
vertex labelled by the letter $v$.  Then we set $m_d=v_{j+1}-u_{i}$. We then
associate to $Q$ the polynomial $P_Q$ defined as the product of $m_d$
over its diagonals.

As an example, Figure \ref{fig:P3-P} shows the polynomials associated 
to the quadrangulations of Figure \ref{fig:P3}, in the corresponding positions.
\begin{figure}[ht!] 
  \centering \includegraphics[scale=1]{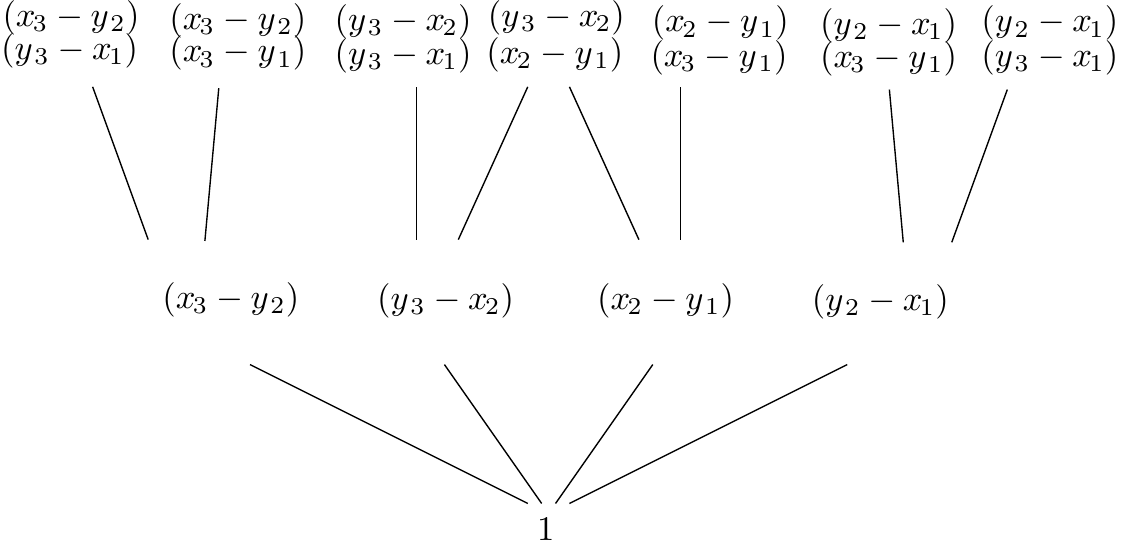} 
  \caption{The polynomials associated to quadrangulations of Figure \ref{fig:P3}}
  \label{fig:P3-P}
\end{figure}

\medskip

To deal with leading terms of polynomials, we will use the
lexicographic order induced by the ordering of the variables:
\begin{equation}
  x_1<y_1<\dots<\omega_1<x_2<y_2<\cdots<x_n<y_n<\dots<\omega_n.  
\end{equation}
The lexicographic order is defined on monomials as follows:
$\A_v<_{\rm lex}\A_w$ if and only if the last non-zero entry of
$v-w$ (componentwise) is negative.

Note that the leading monomial of the polynomial $P_Q$ attached to an
$M$-angulation is easily described: in every binomial factor
written $v_{j+1}-u_i$ with $j+1 > i$, keep only the monomial $v_{j+1}$.

\begin{prop}
  The set of leading monomials of the polynomials $P_Q$ when $Q$
  varies over the set of $M$-angulations coincides with the monomial
  basis $\B_{m,n}$.
\end{prop}

\begin{proof} 
  We need a bijection $\Phi$ between $M$-angulations $Q$ and
  $m$-Dyck paths (or rather $m$-Dyck vectors $v$) such that the leading
  monomial of $P_Q$ is equal to $\A_{\Phi(Q)}$.

  The idea to define $\Phi$ is to compose the leading-monomial application
  $Q\mapsto LM(P_Q)$ with the bijection between monomials and $m$-Dyck
  paths described in section \ref{dyck_paths}.

  Let us instead define the reverse bijection $\Psi$.

  Let $v$ be an $m$-Dyck vector of size $n$. We start from the empty set of
  diagonals on the $m n+2$-gon. We shall add iteratively
  diagonals. Let $D$ denote the current set of diagonals, under
  construction. We read the $m$-vector $v$ from left to right. To any
  non-zero entry $c$ associated to variable $z_k$ we add to $D$ a fan
  with $c$ diagonals as follows. Let $t$ be the letter before $z$ in
  the cyclic order $x < y < \dots < \omega < x$. The (common) ending
  point of the added fan is the first vertex labelled $z$ that comes
  counter-clockwise after the diagonal $k - 1$ of $Q_0$, and the
  starting points are the $c$ last vertices labelled $t$ that are
  available before the ending point. Here being available means not
  being separated from the ending point by a diagonal already in $D$.


  See Figure \ref{fig:bij-ex} for an example of this construction.
  \begin{figure}[ht!] 
    \centering \includegraphics[scale=0.8]{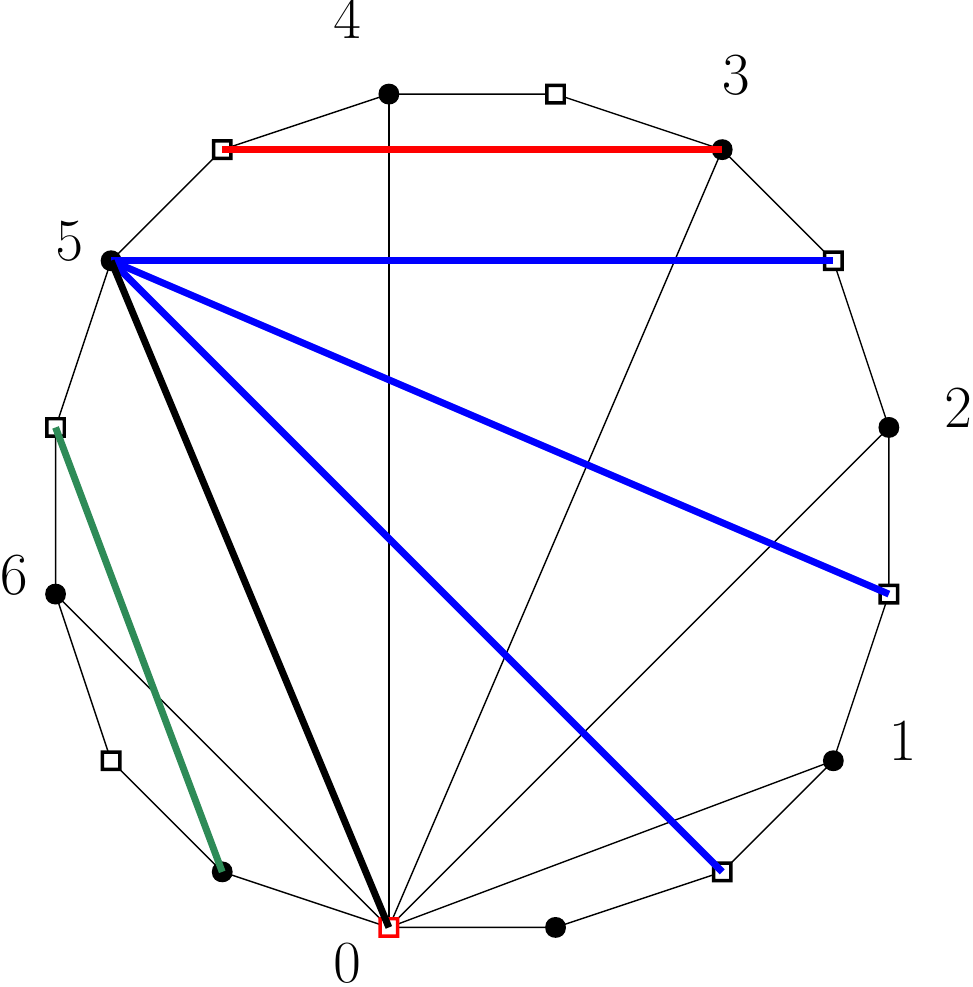} 
    \caption{The quadrangulation $Q$ associated to the monomial 
      $
      \textcolor{red}{x_5}
      \,
      \textcolor{blue}{y_5^3}
      \,
      \textcolor{green2}{y_7}
      $}
 with $P_Q = (x_5-y_4)(y_5-x_3)(y_5-x_2)(y_5-x_1)(y_7-x_6)$.
    \label{fig:bij-ex}
  \end{figure}

  Now the key point is that, when adding such a fan, the starting
  points are always strictly between the apex $0$ and the ending
  point in counter-clockwise order. This is because of the $m$-Dyck
  word property of $v$: for any $1\le \ell\le m n$, one has
  \begin{equation*}
    m (v_1+v_2+\cdots+v_{\ell}) < \ell.
  \end{equation*}
  Indeed, when numbering the vertices of the polygon counter-clockwise
  from the apex $0$ to $m n+1$, the index of the first vertex $t$ in
  counter-clockwise order for the fan added at step $\ell$ is exactly
  $(1+\ell)-m(v_1+v_2+\cdots+v_\ell)-1$.

  The leading term of the product of binomials $m_d$ attached to the
  diagonals in the fan just added is therefore the monomial $z_k^c$,
  by definition of the lexicographic order.

  At the end of the process, there remains only to add some diagonals
  of $Q_0$ to obtain an $M$-angulation. There is a unique way to do
  that. This does not change the product of $m_d$ over all diagonals.

  One has therefore defined a map $\Psi$ from $m$-Dyck words to
  $M$-angulations. This has the property that the leading monomial of
  $P_{\Psi(v)}$ is equal to $\A_v$. Because one can recover the
  $m$-Dyck word $v$ from $\A_v$, this map is injective.

  This proves that $\Psi$ is an injection between sets of equal
  cardinalities, hence a bijection.
\end{proof}

The reverse bijection $\Phi$ can be described as first removing from
$Q$ all diagonals in $Q_0$, then proceeding by successive removals of
fans.

\begin{prop}
  The set of polynomials $P_Q$ (when $Q$
  varies over the set of $M$-angulations) endowed with the poset structure 
  given by divisibility is isomorphic to $\P_{m,n}$.
\end{prop}
\begin{proof}
  The map $Q \mapsto P_Q$ is a bijection, with inverse obtained from
  the factorisation of the polynomial. The flip in $Q$ of a diagonal
  of $Q_0$ corresponds to the multiplication of $P_Q$ by a
  linear binomial.
\end{proof}

Remark: the fan $Q_0$ has an obvious involutive symmetry that fixes
the apex $0$ and flips the ambient polygon. Given the labelling of the
polygon by $\M = \{x,y,\dots,\omega\}$, this involution acts on
variables by simultaneous substitution of letters
\begin{equation*}
  \omega, x, y, z, \ldots \longleftrightarrow \omega, \dots, z, y, x.
\end{equation*}
and renumbering of indices $i \leftrightarrow n+1-i$. By construction,
the set of polynomials $P_Q$ is sent to itself (up to signs) by this
involution acting on variables.

\section{Enumerative aspects}

\subsection{Recursive description}

Let $T_m = \sum_{n \geq 1} \# \Q_{m,n} x^{n}$ be the generating series
for $M$-angulations according to their size $n$. Note that the power
of $x$ is chosen to correspond to the number of regions.

There is a classical recursive decomposition of $M$-angulations,
which describes them as an $M$-angle with an $M$-angulation (or nothing) grafted
on all but one sides. This implies that
\begin{equation}
  \label{eq_T}
  T_m = x (1 + T_m)^{m+1}.
\end{equation}

Recall that the final $M$-angulations are those which
do not contain any diagonal of $Q_0$. They are the maximal elements of the
posets $\P_{m,n}$.

Let $F_m$ be the generating series for final $M$-angulations according
to the number of regions. These objects can be decomposed as one
$M$-angle (the unique region having $0$ in its boundary), on which one
can graft any $M$-angulation on all sides that do not contain the
vertex $0$. One obtains that
\begin{equation}
  \label{eq_F}
  F_m = x (1 + T_m)^m = T_m / (1+T_m).
\end{equation}

Let us now describe another simple decomposition of $M$-angulations,
corresponding to the rightmost expression in the equation \eqref{eq_F}.

To every $M$-angulation $Q$, one can associate a list $\li(Q)$ of
final $M$-angulations, obtained by cutting $Q$ along its initial
diagonals. Let us assume that these pieces are listed in
counter-clockwise order.

\begin{prop}
  Sending $Q$ to $\li(Q)$ defines a bijection from $M$-angulations
  to lists of final $M$-angulations.
\end{prop}

The inverse bijection is very simple: given a list of final
$M$-angulations, one can glue them back along their sides into one
single $M$-angulation.

This inverse bijection from $\li(Q)$ to $Q$ can also be interpreted in
the following way, that will be useful later. Given a list of $k$
final $M$-angulations, one considers the $M$-angulation $Q_0$ with $k$
regions. One then replaces the regions of $Q_0$ by the final
$M$-angulations, in the counter-clockwise order. In the resulting
$M$-angulation $Q$, the initial $Q_0$ can be identified with the union
of all regions that are adjacent to the vertex $0$.

  \begin{figure}[ht!] 
    \centering \includegraphics[scale=0.6]{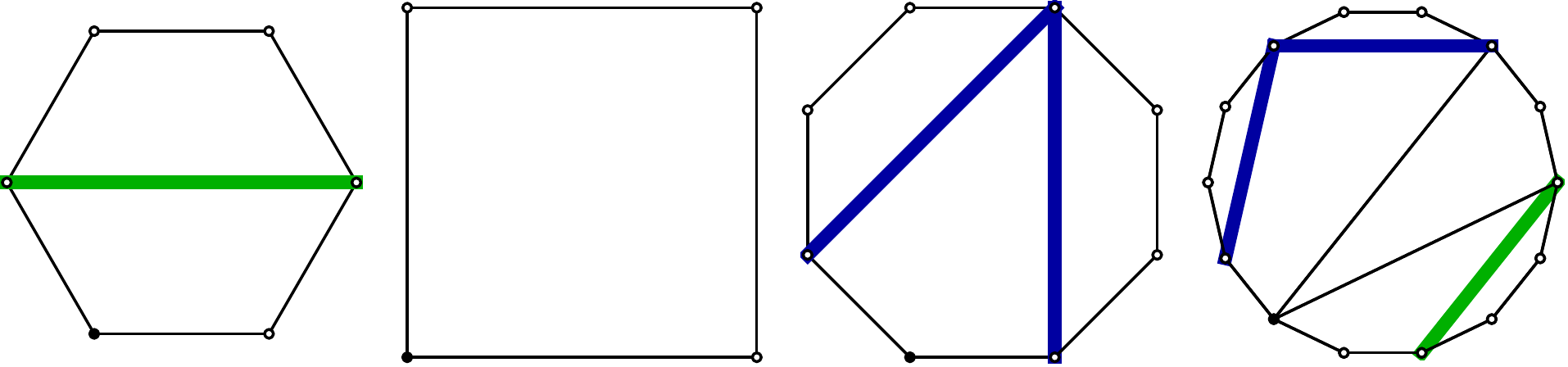} 
    \caption{Illustration of construction $\li$}
    \label{fig:expl-L}
  \end{figure}

\subsection{Rank generating function}

Let $G_m$ be the generating function for the elements of all posets
$\P_{m,n}$, according to their size $n$ and their rank, namely
\begin{equation}
  G_m = \sum_{n\geq 1} x^n \sum_{Q \in \P_{m,n}} z^{\rk{Q}}.
\end{equation}

The generating function $G_m$ satisfies
\begin{equation}
  \label{eq_G}
  G_m = \frac{F_m(z x)/z}{1-F_m(z x)/z}
\end{equation}
where $F_m$ is the generating series for the final elements. Indeed,
any $M$-angulation can be written uniquely as a list of final
$M$-angulations by cutting along the diagonal edges shared with $Q_0$. The
rank parameter is multiplicative along this decomposition. And the
rank of final $M$-angulations is just their size minus $1$, so that
the generating series of final $M$-angulations according to size and
rank is just $F_m(z x)/z$.

Using Lagrange inversion followed by a simple summation of binomial
coefficients, one can deduce from \eqref{eq_T}, \eqref{eq_F} and \eqref{eq_G} that the
rank generating function of $\P_{m,n}$ is given by
\begin{equation}
  \sum_{k=0}^{n-1}\frac{n - k}{n} \binom{m n + k - 1}{k} z^k.
\end{equation}

For $m=1$, this enumeration according to rank was already done in \cite[\S 3]{csar14}.

\subsection{Decomposition of intervals}

Let us now study the intervals in the posets $\P_{m,n}$.

\smallskip

An interval $A$ is a pair of elements $(A^-,A^+)$ in $\P_{m,n}$ that
satisfies $A^- \leq A^+$. In every finite poset, the number of intervals
is also the dimension of the incidence algebra.

Let us call an interval $A$ \textit{initial} if its lower bound $A^-$
is $Q_0$. The set of initial intervals can be identified with the set
$\Q_{m,n}$ of $M$-angulations in $\P_{m,n}$. Indeed, $Q_0$ is smaller
than all $M$-angulations by Lemma \ref{descente_possible}.

Let $\sI_{m,n}$ be the set of intervals in $\P_{m,n}$. Let $I_m = \sum_{n \geq 1} \# \sI_{m,n} x^{n}$ be the generating series for intervals. 

One can get a recursive decomposition for intervals, similar to the previous
decomposition for $M$-angulations. For this, one needs the following
construction.

Suppose that one has an $M$-angulation $B_0$ with $k$ regions, and a
list of $k$ final $M$-angulations $B_1, \dots, B_k$. From this data,
one can build an $M$-angulation $G(B_0 ; B_1, \dots, B_k)$ as
follows. First build the $M$-angulation
$\li^{-1}(B_1, \dots, B_k)$. Removing its $k - 1$ initial
diagonals creates a region with $m k + 2$ sides. Place $B_0$ inside
this region to define the $M$-angulation $G(B_0 ; B_1, \dots, B_k)$.
See Figure \ref{fig:expl-G} for an example.

If $B_0$ is $Q_0$, the construction $G$ is just the inverse of the $\li$ map.

  \begin{figure}[ht!] 
    \centering \includegraphics[scale=0.6]{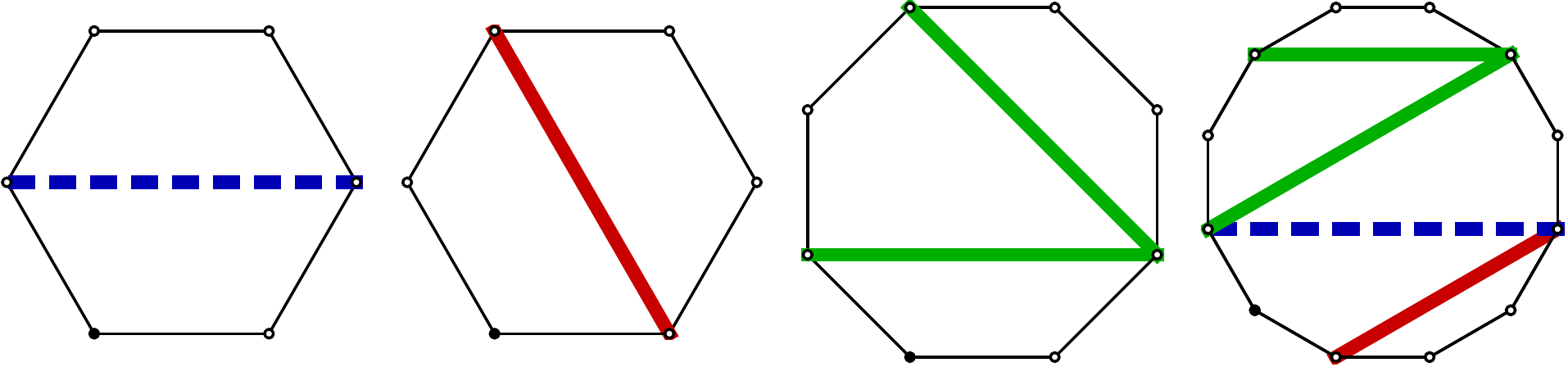} 
    \caption{ $B_0$, $B_1$, $B_2$ and $G(B_0; B_1, B_2)$}
    \label{fig:expl-G}
  \end{figure}

\medskip

Given an interval $A$, one can apply the map $\li$ to its
bottom element $A^-$. This gives a list of final $M$-angulations
$A^-_1, \dots, A^-_k$, such that
\begin{equation}
  A^- = G(Q_0; A^-_1, \dots, A^-_k).
\end{equation}
\begin{prop}
  \label{top_decompo}
  There exists a unique $M$-angulation $A^+_0$ such that
  \begin{equation}
    A^+ = G(A^+_0; A^-_1, \dots, A^-_k).
  \end{equation}
\end{prop}
\begin{proof}
  Uniqueness is clear by the definition of the construction $G$.

  Existence is proved by induction on the difference of initial
  diagonals in $A^-$ and $A^+$. If $A^+ = A^-$, then the only possible
  choice is $A^+_0 = Q_0$.

  Otherwise, let us consider a cover relation $A^- \leq Q' \triangleleft
  Q''$.
  Assume that $Q' = G(Q'_0; A^-_1, \dots, A^-_k)$ by induction.
  Because the cover relations flips an initial diagonal $d$ and does
  not change the other ones, the diagonal $d$ must in fact belong to
  $Q'_0$. Therefore one can flip it in $Q'_0$ to get $Q''_0$ such that
  $Q'' = G(Q''_0; A^-_1, \dots, A^-_k)$.
\end{proof}

Keeping the same notations, one also has the following result.
\begin{prop}
  \label{reduction_int}
  Every element of the interval $[A^-,A^+]$ can be uniquely written as
  $G(Q; A^-_1, \dots, A^-_k)$ for some $Q$ in $[Q_0, A_0^+]$.  The
  interval $A$ is isomorphic to the interval $[Q_0, A_0^+]$.
\end{prop}
\begin{proof}
  This result follows from the proof of Prop. \ref{top_decompo}. In
  fact, more is true: all elements greater than $A^-$ can be uniquely written
  $G(Q; A^-_1, \dots, A^-_k)$ for some $Q$ and this bijection identifies the
  upper ideal of $A^-$ with a smaller poset of type $\P$.
\end{proof}

Proposition \ref{top_decompo} implies the following decomposition.
\begin{prop}
  The map sending an interval $A$ to the pair $(A^+_0, \li(A^-))$  
  defines a bijection between intervals and pairs
  $(B_0 , (B_1, \dots, B_k))$ where $B_0$ is an $M$-angulation with $k$
  regions and $B_1, \dots, B_k$ are final $M$-angulations.
\end{prop}
The inverse bijection is given by
\begin{equation}
  A^- = G(Q_0;  B_1, \dots, B_k) \quad\text{and}\quad  A^+ = G(B_0;  B_1, \dots, B_k).
\end{equation}

\begin{coro}
  The generating series of intervals can be expressed using
  those of $M$-angulations and final $M$-angulations as
  \begin{equation}
    \label{eq_I}
    I_m = T_m(F_m) = T_m(x (1 + T_m)^m).
  \end{equation}
\end{coro}
The second equality follows from \eqref{eq_F}.

\section{Isomorphism types of intervals}

The aim of this section is to give a description of the intervals as posets, and to prove in particular that they are distributive lattices with Möbius numbers in $\{-1,0,1\}$.

By Proposition \ref{reduction_int}, every interval is isomorphic to an
initial interval. It is therefore enough to study initial intervals.

Let us call an interval \textit{initial-final} if its minimum is $Q_0$ and its
maximum is a final quadrangulation.

\begin{prop}
  Every initial interval is isomorphic to a product of initial-final intervals.
\end{prop}
\begin{proof}
  Let $A=[Q_0,A^+]$ be an initial interval. Consider the set of
  diagonals of $Q_0$ that belong to $A^+$. Then one can cut both $Q_0$
  and $A^+$ along these diagonals. Every piece $A_i^-$ of $Q_0$ is a smaller
  $Q_0$. Every piece $A_i^+$ of $A^+$ is a final quadrangulation. By definition of the
  partial order, the same diagonals belong to every element of $A$. Cutting along
  them gives an isomorphism with the product of the intervals $[A_i^-,A_i^+]$, which are
  initial-final intervals.
\end{proof}

Let us now proceed to a more subtle decomposition.

Let $Q$ be a final triangulation. Let $R_0$ be the unique region in
$Q$ with $0$ in its boundary. Removing $R_0$ from $Q$ leaves a certain
number $k$ of polygons, that will be called the \textit{blocks}. Let
us call $k$ the \textit{width} of $Q$. The width can be $0$ only if
$Q$ is reduced to $R_0$. Otherwise it lies between $1$ and $m$.

\begin{prop}
  Every initial-final interval $A$ is isomorphic to a product of
  initial-final intervals of width $1$.
\end{prop}
\begin{proof}
  The main idea is that the flips downwards from $Q$ happen completely
  independently in distinct blocks. Let now us give a detailed
  argument for this independence.

  Let us consider the set $D$ of diagonals that go from vertex $0$ to
  one of the vertices of the region $R_0$, except the two vertices
  that are neighbors of $0$. These diagonals are inside the region
  $R_0$ and therefore do not belong to $Q$. The extremities of these
  diagonals receive exactly once every label from $0$ to $m-2$ in
  clockwise order, and therefore never get the label $m-1$. Note that
  the extremities of the initial diagonal are labeled by $m-1$.

  Let us show by induction downward from $Q$ that for every
  $Q' \in [Q_0,Q]$, the diagonals in $D$ do not belong to $Q'$ and do
  not cross any diagonal of $Q'$. This is clear for $Q$. The effect of
  a down flip is to replace a diagonal by an initial diagonal. This
  initial diagonal is not in $D$ and does not cross any diagonal in
  $Q$, because of the labeling of its extremity is $m-1$.

  One can therefore cut along the diagonals in $D$. Replacing every
  block but a fixed one by a trivial block gives a map to an initial-final
  interval of width $1$. Taking the product of all these maps gives
  the desired isomorphism.

\end{proof}

\begin{prop}
  The number of elements covered by the maximum in a initial-final
  interval $[Q_0,Q]$ is the width of $Q$.
\end{prop}
\begin{proof}
  In the proof of Lemma \ref{descente_possible}, it was shown that in every
  polygon of $Q$ minus $R_0$, there is exactly one initial diagonal
  $d$ that cross just one diagonal $d'$ of $Q$. The set of diagonals
  of $Q$ that can be flipped down in the poset is exactly the
  collection of these $d'$, and their number is therefore the width of
  $Q$.
\end{proof}

In particular, initial-final intervals of width $1$ have a maximum
that covers a unique element.

\begin{coro}
  The intervals in $\P_{m,n}$ can be build iteratively by either
  adding a maximum to a smaller one or by taking a product of several
  smaller ones.
\end{coro}
\begin{proof}
  By induction on both the height (difference of ranks) of the
  intervals and the size of the ambient polygon. Every interval is a
  product of initial-final ones. Every initial-final interval is a
  product of initial-final intervals of width $1$. In both cases, if
  the product is not reduced to one element, the factors live in
  strictly smaller polygons.

  Every initial-final interval of width $1$ is obtained by adding a maximum
  to an interval, which has smaller height.
\end{proof}

Let us define a \textit{forest poset} as a poset where every element
is covered by at most one element. These posets can be build
iteratively from the empty poset by two operations, namely adding a
maximum element and taking the disjoint union. Their Hasse diagrams are forests of rooted trees, where the roots are the maximal elements.

\begin{prop}
  Every interval in $\P_{m,n}$ is a distributive lattice, isomorphic
  to the lattice of order ideals of a forest poset.
\end{prop}
\begin{proof}
  By the results above, the intervals in $\P_{m,n}$ can be obtained
  from smaller intervals using two operations, namely adding a maximum
  and taking a product. These two operations correspond to adding a
  maximum or taking the disjoint union, on the poset of
  join-irreducible elements. The statement follows by induction.
\end{proof}

For $m=1$, this property of intervals was already obtained in \cite[\S 2]{csar14}.

\begin{coro}
  \label{moeb}
  Möbius numbers of intervals in posets $\P_{m,n}$ belong to $\{-1,0,1\}$.
\end{coro}
\begin{proof}
  This property is preserved by adding a maximum or taking a product.
\end{proof}

\bibliographystyle{alpha}
\bibliography{note_aval_chap}

\begin{thebibliography}{STW15}

\bibitem[AB03]{ab2003}
J.-C. Aval and N.~Bergeron.
\newblock Catalan paths and quasi-symmetric functions.
\newblock {\em Proc. Amer. Math. Soc.}, 131(4):1053--1062, 2003.

\bibitem[ABB04]{abb2004}
J.-C. Aval, F.~Bergeron, and N.~Bergeron.
\newblock Ideals of quasi-symmetric functions and super-covariant polynomials
  for {$S_n$}.
\newblock {\em Adv. Math.}, 181(2):353--367, 2004.

\bibitem[Ava07]{aval06}
J.-C. Aval.
\newblock Ideals and quotients of {$B$}-quasisymmetric polynomials.
\newblock {\em S\'em. Lothar. Combin.}, 54:Art. B54d, 13, 2005/07.

\bibitem[BH08]{BH08}
P.~Baumann and C.~Hohlweg.
\newblock A {S}olomon descent theory for the wreath products {$G\wr S_n$}.
\newblock {\em Trans. Amer. Math. Soc.}, 360(3):1475--1538, 2008.

\bibitem[BPR12]{bpr2012}
F.~Bergeron and L.-F. Pr{\'e}ville-Ratelle.
\newblock Higher trivariate diagonal harmonics via generalized {T}amari posets.
\newblock {\em J. Comb.}, 3(3):317--341, 2012.

\bibitem[Cha05]{chapoton05}
F.~Chapoton.
\newblock Une base sym\'etrique de l'alg\`ebre des coinvariants
  quasi-sym\'etriques.
\newblock {\em Electron. J. Combin.}, 12:Note 16, 7, 2005.

\bibitem[CSS12]{csar12}
S.~A. Csar, R.~Sengupta, and W.~Suksompong.
\newblock On a subposet of the {T}amari lattice.
\newblock In {\em 24th International Conference on Formal Power Series and
  Algebraic Combinatorics (FPSAC 2012)}, DMTCS Proceedings, pages 563--574,
  2012.

\bibitem[CSS14]{csar14}
S.~A. Csar, R.~Sengupta, and W.~Suksompong.
\newblock On a subposet of the {T}amari lattice.
\newblock {\em Order}, 31(3):337--363, 2014.

\bibitem[Ges84]{gessel84}
Ira~M. Gessel.
\newblock Multipartite {$P$}-partitions and inner products of skew {S}chur
  functions.
\newblock In {\em Combinatorics and algebra ({B}oulder, {C}olo., 1983)},
  volume~34 of {\em Contemp. Math.}, pages 289--317. Amer. Math. Soc.,
  Providence, RI, 1984.

\bibitem[Pal03]{pallo03}
J.~M. Pallo.
\newblock Right-arm rotation distance between binary trees.
\newblock {\em Inform. Process. Lett.}, 87(4):173--177, 2003.

\bibitem[Poi98]{poirier1998}
S.~Poirier.
\newblock Cycle type and descent set in wreath products.
\newblock In {\em Proceedings of the 7th {C}onference on {F}ormal {P}ower
  {S}eries and {A}lgebraic {C}ombinatorics ({N}oisy-le-{G}rand, 1995)}, volume
  180, pages 315--343, 1998.

\bibitem[STW15]{stw2015}
C.~{Stump}, H.~{Thomas}, and N.~{Williams}.
\newblock {Cataland: Why the Fuss?}
\newblock {\em ArXiv e-prints}, March 2015.

\end{thebibliography}

{Fr\'ed\'eric Chapoton} \\
{Institut de Recherche Mathématique Avancée, CNRS UMR 7501, Université de Strasbourg, F-67084 Strasbourg Cedex, France} \\
{chapoton@unistra.fr}

{Jean-Christophe Aval} \\
{LaBRI, Université de Bordeaux, 351, cours de la Libération, 33405 Talence Cedex, France} \\
{jean-christophe.aval@labri.fr}

\end{document}